\documentclass{amsart}

\usepackage{footmisc}
\usepackage{bigints}
\usepackage{bbm}
\usepackage{geometry}
\usepackage{amssymb}
\usepackage{enumerate}
\usepackage{mathrsfs}
\usepackage{hyperref}
\usepackage[justification=centering, skip=20pt]{caption}
\usepackage[all,cmtip]{xy}
\usepackage{graphicx}
\usepackage{dynkin-diagrams}
\usepackage{setspace}

\hypersetup{
    colorlinks,
    citecolor=black,
    filecolor=black,
    linkcolor=black,
    urlcolor=black
    }

\newcommand{\R}[0]{\mathbb{R}}

\newtheorem{thm}{Theorem}[section]

\newtheorem{prop}[thm]{Proposition}
\newtheorem{lem}[thm]{Lemma}
\newtheorem{conj}[thm]{Conjecture}

\newtheorem*{defn*}{Definition}

\newtheorem{rem}[thm]{Remark}

\let\emptyset\varnothing

\numberwithin{equation}{section}

\usepackage{etoolbox}

\makeatletter
\patchcmd{\@settitle}{\uppercasenonmath\@title}{}{}{}
\patchcmd{\@setauthors}{\MakeUppercase}{}{}{}
\patchcmd{\section}{\scshape}{}{}{}
\makeatother

\makeatletter
\@namedef{subjclassname@2020}{\textup{2020} Mathematics Subject Classification}
\makeatother

\title[Nonlinear Schr\"odinger equations on compact product manifolds]{Local well-posedness for nonlinear Schr\"odinger equations on compact product manifolds}

\author{Yunfeng Zhang}


\begin{document}

\onehalfspacing

\begin{abstract}
We prove new local well-posedness results for nonlinear Schr\"odinger equations posed on a general product of spheres and tori, by the standard approach of multi-linear Strichartz estimates. To prove these estimates, we establish multi-linear bounds for the joint spectral projector associated to the Laplace--Beltrami operators on the individual sphere factors of the product manifold. To treat the particular case of the cubic NLS on a product of two spheres at critical regularity, we prove a sharp $L^\infty_xL^p_t$ estimate of the solution to the linear Schr\"odinger equation on the two-torus. 
\end{abstract}

\maketitle

{\let\thefootnote\relax\footnotetext{Y. Zhang: Department of Mathematical Sciences, University of Cincinnati, Cincinnati, OH 45221-0025, USA.\\ email: yunfengzhang108@gmail.com}}

\section{Introduction}

The goal of this paper is to provide new results of local well-posedness for nonlinear Schr\"odinger equations (NLS) posed on a general product of spheres and tori, complementing the results obtained by us in \cite{Zha21} for the cubic NLS, and generalizing them to NLS of all algebraic nonlinearities. 
Let $M$ first be a compact Riemannian manifold of dimension $d$ equipped with the Laplace--Beltrami operator $\Delta$. The nonlinear Schr\"odinger equation of algebraic nonlinearity posed on $M$ reads 
\begin{align}\label{NLS}
i\partial_t u+\Delta u = \pm |u|^{2k}u,\tag{NLS}
\end{align}
where $k$ is a positive integer, and $u=u(t,x)$ is a function of time $t\in \R$ and space $x\in M$. Compared with the standard model where the underlying manifold $M$ is a Euclidean space, NLS posed on a compact Riemannian manifold has a much richer geometric flavor and has attracted a lot of attention \cite{BP11, BGT02, BGT04, BGT05, BGT052, Bou93, BD15, CCT03, GOW14, Her13, HTT11, HTT14, HS15, KV16, SZ25, Wan13, Yan15, Zha21, Zha23, Zha16}. One is usually interested in answering well-posedness questions for initial data lying in an $L^2$-based Sobolev space $H^{s}$ and a natural question is to understand the optimal range of $s$ for which the NLS is well-posed. Pretending $M$ to be a Euclidean space and considering scaling, the critical regularity for \eqref{NLS} is understood to be
$$s_c=\frac{d}2-\frac{1}{k},$$
serving as a guideline for the threshold of well-posedness for NLS on arbitrary manifolds\footnote{Of course, this $s_c$ may not be the actual threshold above which (uniform) well-posedness holds, and below which
(uniform) well-posedness fails. A notable example is the cubic NLS on $\mathbb{S}^2$, for which the threshold regularity for uniform well-posedness is $s=\frac14$ instead of $s_c=0$ \cite{BGT05}. }. 

The choice of products of spheres and tori as the underlying manifolds is motivated not only by their role as the next natural class beyond the well-studied spheres \cite{BGT04, BGT052, Her13} and tori \cite{HTT11, BD15, KV16} for investigating nonlinear Schr\"odinger dynamics on compact manifolds, but also by the fact that they give rise to genuinely new challenges and important open problems. First, the product models require both a synthesis of, and new insights into, existing theories for NLS on spheres and tori. For example, they motivate the study of new multi-linear estimates for the eigenfunctions of the Laplace--Beltrami operator, such as in Theorem \ref{mljspe} below and the recent development on the three-sphere \cite{DZZ25}, and new estimates in Fourier analysis and number theory, such as the anisotropic Strichartz estimates on tori (Conjecture~\ref{rs}). Second, the product models include a lot of physically meaningful cases, such as the energy-critical cubic NLS equations on $\mathbb{S}^2 \times \mathbb{S}^2$ and $\mathbb{S}^3 \times \mathbb{T}$, which still lack a resolution of the critical well-posedness problem (see Remarks~\ref{rem: S2S2} and~\ref{rem: S3T}).
Furthermore, such product spaces provide the simplest prototypes of higher-rank compact symmetric spaces and thus constitute a natural testing ground before addressing NLS on general higher-rank settings, where more sophisticated tools from algebra and microlocal analysis would become indispensable. 


Let us now state the main contributions of this paper. Throughout the paper, we will use $A\lesssim B$ to mean $A \leq CB$ for some
positive constant $C$, and $A\sim B$ to mean $A\lesssim B$ and $B\lesssim A$.

\begin{thm}[Multi-linear Strichartz estimate]\label{mls}
Let $M$ be a product of spheres and tori: $M=\mathbb{S}^{d_1}\times\mathbb{S}^{d_2}\times\cdots\times \mathbb{S}^{d_{r_0}}\times \mathbb{T}^{r_1}$, with $d_i\geq 2$ ($i=1,2,\ldots,r_0$) and $r:=r_0+r_1\geq 2$. Let $r_2$ (respectively, $r_3$) be the number of 2-sphere (respectively, 3-sphere) factors in this product. 
Let $f^j\in L^2(M)$ be spectrally localized to the window $[N_j,2N_j]$ with respect to $\sqrt{-\Delta}$, that is, 
$\mathbbm{1}_{[N_j,2N_j]}(\sqrt{-\Delta})f^j=f^j$, $j=1,2,\ldots,k+1$. Let the spectral parameters be ordered such that $N_1\geq N_2\geq\cdots\geq N_{k+1}\geq 1$. Let $I$ be a fixed time interval. Then: 
\\    
(i) For $k=1$, $r\geq 3$, there exists $\delta>0$ such that 
$$\|e^{it\Delta}f^1\ e^{it\Delta}f^2\|_{L^2(I\times M)}\lesssim \left(\frac{N_2}{N_1}+\frac1N_2\right)^{\delta}N_2^{\frac{d}2-1+\frac{r_2}{4}}(\log N_2)^{\frac{r_3}{2}}\|f^1\|_{L^2(M)}\|f^2\|_{L^2(M)}.$$
(ii) For $k=1$, $r=2$, we have for all $\varepsilon>0$ 
$$\|e^{it\Delta}f^1\ e^{it\Delta}f^2\|_{L^2(I\times M)}\lesssim  N_2^{\frac{d}2-1+\frac{r_2}{4}+\varepsilon}(\log N_2)^{\frac{r_3}{2}}\|f^1\|_{L^2(M)}\|f^2\|_{L^2(M)}.$$
(iii) For $k=1$ and the special case $M=\mathbb{S}^{d_1}\times\mathbb{S}^{d_2}$, $d_1,d_2\geq 4$, 
there exists $\delta >0$ such that
$$\|e^{it\Delta}f^1\ e^{it\Delta}f^2\|_{L^2(I\times M)}\lesssim \left(\frac{N_2}{N_1}+\frac1{N_2}\right)^\delta N_2^{\frac{d}2-1}\|f^1\|_{L^2(M)}\|f^2\|_{L^2(M)}.$$
(iv) For $k\geq 2$, $r\geq 3$, there exists $\delta_0>0$ such that for all $\delta\in [0,\delta_0)$ and $\eta>0$
\begin{align*}
\left\|\prod_{j=1}^{k+1}e^{it\Delta}f^j\right\|_{L^2(I\times M)}&\lesssim \left(\frac{N_{k+1}}{N_1}+\frac{1}{N_2}\right)^{\delta}
N_2^{\frac{d}{2}-1+\frac{r_2}4+r_3\eta+\delta(k-1)}N_3^{\frac{d}{2}-\frac{r_2}4-r_3\eta-\delta}\\
&\ \ \ \ \ \ \cdot \prod_{j=4}^{k+1}N_j^{\frac{d}{2}-\delta}\prod_{j=1}^{k+1}\|f^j\|_{L^2(M)}.
\end{align*}
(v) For $k\geq 2$, $r=2$, there exists $\delta_0>0$ such that for all $\delta\in [0,\delta_0)$, $\eta>0$ and $\varepsilon>0$
\begin{align*}
\left\|\prod_{j=1}^{k+1}e^{it\Delta}f^j\right\|_{L^2(I\times M)}&\lesssim \left(\frac{N_{k+1}}{N_1}+\frac1{N_2}\right)^{\delta}
N_2^{\frac{d}{2}-1+\frac{r_2}4+r_3\eta+\varepsilon+\delta(k-1)}\\
&\ \ \ \ \ \ \cdot N_3^{\frac{d}{2}-\frac{r_2}4-r_3\eta-\varepsilon-\delta}\prod_{j=4}^{k+1}N_j^{\frac{d}{2}-\delta}\prod_{j=1}^{k+1}\|f^j\|_{L^2(M)}.
\end{align*}
\end{thm}

As a corollary, we have the following new local well-posedness results. 

\begin{thm}[Main result]\label{lwp}
Let $M=\mathbb{S}^{d_1}\times\mathbb{S}^{d_2}\times\cdots\times \mathbb{S}^{d_{r_0}}\times \mathbb{T}^{r_1}$, $r=r_0+r_1\geq 2$, $d_i\geq 2$, $i=1,2,\ldots,r_0$. Let $r_2$ (respectively, $r_3$) be the number of 2-sphere (respectively, 3-sphere) factors in this product. Then the \eqref{NLS} is locally well-posed with initial data $u(0,x)\in H^s$ for:
\\
(1) $s\geq s_c$ (critical), when:
    \begin{itemize}

    \item  $r_2=0,1$, $k\geq 2$ (with the case of the quintic NLS on $\mathbb{S}^2\times\mathbb{T}^1$ already treated in \cite{HS15});    
    \item  $M=\mathbb{S}^{d_1}\times\mathbb{S}^{d_2}$, $d_1,d_2\geq 4$, $k\geq 1$;   
    
     \item  $r_2=2$, $k\geq 3$;   
      \item  $r_2=3$, $k\geq 5$;       
      
    \end{itemize}
(2) $s>s_c$ (almost critical), when: 
\begin{itemize}

\item  $r_2=2$, $r=2,3$, $k= 2$;

    \end{itemize}
(3) $s>s_0>s_c$ (sub-critical), when: 
\begin{itemize}
    
    \item  $r_2=1$, $r\leq 11$, $k= 1$, $s_0=\frac{d}2-\frac{3}4$.    
    \end{itemize}

\end{thm}

\begin{rem}
In the above theorem, we only included results with the range of regularities that has not been reached by previous literature. We also did not include any results treatable by currently known linear Strichartz estimates, which however will be discussed and summarized in the Appendix. 
For a summary of current status of local well-posedness of \eqref{NLS} on compact manifolds, see Table \ref{lit}. 
\end{rem}

\begin{rem}\label{rem: S2S2}
The sub-critical ranges listed in Table \ref{lit}, including those in (3) of Theorem~\ref{lwp}, 
as well as the multi-linear Strichartz estimates in Theorem \ref{mls}, are not expected to be sharp in general. We currently do not yet have a good understanding of NLS on product manifolds with 2-sphere factors. 
As an example, for the cubic NLS on $\mathbb{S}^2\times\mathbb{S}^2$, our approach gives well-posedness for $s\geq \frac32$, but there is reason to expect that well-posedness may hold at the critical regularity $s_c=1$. This optimism is mainly motivated by the recent work \cite{HS25} of Huang--Sogge, where they successfully established sharp (up to $\varepsilon$-losses) Strichartz estimates on $\mathbb{S}^2$. This provides evidence that scale-invariant $L^p$-Strichartz estimate may in fact hold on $\mathbb{S}^2\times\mathbb{S}^2$ for some $p<4$; see the discussion in the final section of \cite{DZZ25}. 

\end{rem}

\begin{center}

\begin{table}
\small
{\renewcommand{\arraystretch}{1.2}
\begin{tabular}{ccc}
\hline  
\begin{tabular}{cc}
Compact manifold 
of dimension $d$
\end{tabular}& \vline & Known local well-posedness\\
\hline  
General & \vline & $s>\frac d2-\frac1{2k}>s_c=\frac{d}2-\frac1k$, $k\geq 1$ \cite{BGT04}\\
\hline 
Tori $\mathbb{T}^d$ \footnotemark[1]{}

  & \vline & 
\begin{tabular}{ll}
$d=1$, $k=1$:& $s\geq 0>s_c=-\frac12$ \cite{Bou93}\\
$d=1$, $k\geq 2$:& $s>s_c$ \cite{Bou93}\\
$d=1$, $k\geq 3$:& $s\geq s_c$ \cite{HTT11, Wan13}\\
$d=2$, $k\geq 2$:& $s>s_c$ \cite{Bou93, GOW14}\\
$d\geq 3$, $k\geq 1$:& $s\geq s_c$ \cite{HTT11, HTT14, Wan13, GOW14, BD15, KV16}
\end{tabular}
\\

\hline 
\begin{tabular}{c}
     Spheres $\mathbb{S}^d$  \\
    (and Zoll manifolds) 
\end{tabular}
 & \vline & 
\begin{tabular}{ll}
$d=2$, $k=1$: & $s>\frac14>s_c=0$ \cite{BGT05}\\
$d=2$, $k\geq 2$: & $s>s_c$  \cite{Yan15}\\
$d=2$, $k\geq 3$: & $s\geq s_c$  \cite{Zha16}\\
$d\geq 3$, $k\geq 1$: & $s> s_c$ \cite{Yan15}\\
$d\geq 3$, $k\geq 2$: &$s\geq s_c$ \cite{Her13, Zha16}
\end{tabular}
\\

\hline 
\begin{tabular}{cc}
$\mathbb{S}^{d_1}\times\mathbb{S}^{d_2}\times\cdots\times\mathbb{S}^{d_{r_0}}\times\mathbb{T}^{r_1}$ \footnotemark[2]\\
$d_i\geq 2$, $i=1,\ldots,r_0$\\
$r:=r_0+r_1\geq 2$\\
$r_2:=$ number of 2-sphere factors\\
$r_3:=$ number of 3-sphere factors
\end{tabular}
& \vline & 
\begin{tabular}{ll}
$r_2=0$, $k\geq 1$: & $s> s_c$ \cite{BGT04, Zha21} [Z]\\
$\mathbb{S}^3\times\mathbb{T}^{r_1}$, $r_1\geq 2$, $k=1$: & $s\geq s_c$ \cite{DZZ25}\\
$\mathbb{S}^{d_1}\times \mathbb{S}^{d_2}$, $d_1,d_2\geq 4$, $k= 1$: & $s\geq s_c$ [Z]\\
$r_2=r_3=0$, $r\geq 3$, $k= 1$: & $s\geq s_c$ \cite{Zha21}\\
$r_2=0$, $k\geq 2$: & $s\geq s_c$ [Z]\\
$r_2=1$, $r\leq 11$, $k=1$: & $s> \frac{d}2-\frac{3}4>s_c=\frac{d}2-1$ [Z]\\
$r_2=1$, $r\geq 12$, $k=1$: & $s> \frac{d}2-\frac{r}{r+4}>s_c=\frac{d}2-1$ \cite{Zha21}\\
$r_2=1$, $k\geq 2$: & $s\geq s_c$ \cite{HS15} [Z]\\
$r_2=2$, $r\leq 4$, $k=1$: & $s\geq \frac d2-\frac12>s_c=\frac{d}2-1$ \cite{BGT04}\\
$r_2=2$, $r\geq 5$, $k=1$: & $s\geq \frac d2-\frac{r}{r+4}>s_c=\frac{d}2-1$ \cite{Zha21}\\
$r_2=2$, $k\geq 2$: & $s> s_c$ \cite{Zha21}, [Z]\\
$r_2=2$, $k\geq 3$: & $s\geq s_c$ [Z]\\
$r_2=3$, $r\leq 4$, $k=1$: & $s>\frac d2-\frac12>s_c=\frac d2-1$ \cite{BGT04}\\
$r_2=3$, $r\geq 5$, $k=1$: & $s>\frac d2-\frac {r}{r+4} >s_c=\frac d2-1$ \cite{Zha21}\\
$\mathbb{S}^2\times\mathbb{S}^2\times\mathbb{S}^2$, $k=2$: & $s>\frac d2-\frac37>s_c=\frac d2-\frac12$ \cite{Zha21}\\
$r_2=3$, $r\geq4$, $k=2$: & $s>s_c=\frac d2-\frac12$ \cite{Zha21}\\
$r_2=3$, $k\geq 3$: & $s>s_c$ \cite{BGT04}\\
$r_2=3$, $k\geq 5$: & $s\geq s_c$ [Z]\\
$r_2\geq 4$, $k=1$: & $s>\frac d2-\frac {r}{r+4} >s_c=\frac{d}2-1$ \cite{Zha21}\\
$r_2\geq 4$, $k\geq 2$: & $s>s_c$ \cite{Zha21}
\end{tabular}
\\

\hline 
\begin{tabular}{cc}
Symmetric spaces of compact type \\
of rank $r\geq 2$ 
\end{tabular}
& \vline &
\begin{tabular}{cc}
$r=2,3$, $k\geq 3$:& $s>s_c$ \cite{Zha21}\\

Rank-3 compact simple Lie groups \footnotemark[3], $k\geq 2$:
&$s>s_c$  \cite{Zha23}\\
$r\geq 4$, $k\geq 2$: & $s>s_c$ \cite{Zha21}\\
\end{tabular}
\\

\hline

\end{tabular}
}
\caption{Local well-posedness of \eqref{NLS} on compact manifolds ([Z]: the current paper)}
\label{lit}

\end{table}
\end{center}

\footnotetext[1]{Works for both rational and irrational rectangular tori. All the local well-posedness results of subcritical regularity work for general non-rectangular flat tori also.}

\footnotetext[2]{The sphere factors can all be replaced by rank-1 compact symmetric spaces of compact type.}

\footnotetext[3]{Consisting of $\text{SU}(4)$, $\text{SO}(6)$, $\text{PSO}(6)$, 
$\text{Spin}(7)$, $\text{SO}(7)$, $\text{Sp}(3)$ and $\text{PSp}(3)$.}

As is clear from the statement of Theorem~\ref{mls}, our approach of proving local well-posedness in Theorem~\ref{lwp} follows the standard one of a multi-linear Strichartz estimate, which was first utilized for NLS in the compact manifold setting by Bourgain \cite{Bou93} for tori, and then by Burq--G\'erard--Tzvetkov \cite{BGT05} for compact surfaces, both to treat sub-critical regularities, and later extended by Herr--Tataru--Tzvetkov \cite{HTT11} to treat critical regularities for the three-torus. Following a similar approach, a lot of subsequent work \cite{BGT052, Her13, HTT11, HTT14, HS15, KV16, Wan13, Yan15, Zha21, Zha16} appeared which provided further local well-posedness results for NLS on compact manifolds. The key ingredient of getting the multi-linear Strichartz estimates in these works is a multi-linear spectral projector estimate associated to the Laplace--Beltrami operator, which we review below:


\begin{thm}[Multi-linear spectral projector estimate of Burq--G\'erard--Tzvetkov \cite{BGT05, BGT052}]\label{mlspe}

Let $M$ be a compact Riemannian manifold of dimension $d\geq 2$ equipped with the Laplace--Beltrami operator $\Delta$. For $N_j\geq 1$, let the spectral projector be defined by 
$\chi_{N_j}=\chi(\sqrt{-\Delta}-N_j)$, where $\chi\in C_c^\infty(\R)$ with $\chi(0)=1$, and $j=1,2,\ldots,k+1$. Let $\delta_2=1$ if $\dim(M)=2$ and 0 otherwise. Let $\delta_3=1$ if $\dim(M)=3$ and 0 otherwise. Assume $N_1\geq N_2\geq \cdots\geq N_{k+1}\geq 1$. 
Then:\\ 
(i) For $k=1$, we have  
$$\left\|\chi_{N_1} f^1\ \chi_{N_2} f^2\right\|_{L^2(M)}\lesssim N_2^{\frac{d-2}{2}+\frac{\delta_2}4}(\log N_2)^{\frac{\delta_3}2} \|f^1\|_{L^2(M)}\|f^2\|_{L^2(M)}.$$
(ii) For $k\geq 2$ and $\eta>0$, we have 
$$\left\|\prod_{j=1}^{k+1}\chi_{N_j} f^j\right\|_{L^2(M)}\lesssim N_2^{\frac{d-2}{2}+\frac{\delta_2}4+\delta_3\eta }N_3^{\frac{d-1}{2}-\frac{\delta_2}4-\delta_3\eta}\prod_{j=4}^{k+1}N_j^{\frac{d-1}{2}}\prod_{j=1}^{k+1}\|f^j\|_{L^2(M)}.$$
\end{thm}

For a compact product manifold such as a product of {\it one} sphere and tori, the above multi-linear spectral projector estimate can be used to provide sharp local well-posedness results for the posed NLS by applying it to the sphere factor. This was done by Herr--Strunk \cite{HS15} for the product of a two-sphere and a circle. However, when there are more than one sphere factors in the product, the above multi-linear spectral projector estimate would lose track of essential geometric information if applied to a product of spheres. Instead, we prove the following alternative spectral projector estimates on a product manifold that respect the product structure, and to do so we must deal with the {\it joint} spectral projector associated with all the individual Laplace--Beltrami operators on the sphere factors:

\begin{thm}[Multi-linear joint spectral projector estimate]\label{mljspe}
Let $M_i$ be a compact Riemannian manifold of dimension $d_i\geq 2$, $i=1,2,\ldots,r$, $r\geq 1$. Let $M=M_1\times M_2\times\cdots \times M_r$ be their product. Let $\Delta_i$ denote the Laplace--Beltrami operator on $M_i$. 
Let $\chi_i\in C_c^\infty(\mathbb{R})$, $\chi_i(0)=1$. 
For $\lambda=(\lambda_1,\ldots,\lambda_r)\in\mathbb{R}^r$, set 
$|\lambda|=\sqrt{\lambda_1^2+\cdots+\lambda_r^2}$. Define the joint spectral projector around $\lambda\in\mathbb{R}^r_{\geq 1}$ as follows
$$\chi_\lambda:=\prod_{i=1}^r\chi_{i}(\sqrt{-\Delta_i}-\lambda_i).$$ 
Now consider $\lambda^j=(\lambda^j_1,\ldots,\lambda^j_r)\in\mathbb{R}^r_{\geq 1}$, $j=1,\ldots,k+1$. 
Suppose the spectral parameters $N_j:=|\lambda^j|$ are ordered such that $N_1\geq N_2\geq\cdots\geq N_{k+1}\geq 1$. Let $r_2$ (respectively, $r_3$) be the number of two-dimensional (respectively, three-dimensional) factors in the product $M_1\times M_2\times \cdots\times M_r$. Then
for 
$$C(N_1,\ldots,N_{k+1})
=\left\{
\begin{array}{ll}
   N_2^{\frac{d-2r}{2}+\frac{r_2}4}(\log N_2)^{\frac{r_3}2},  & \text{ if }k=1, \\
   N_2^{\frac{d-2r}{2}+\frac{r_2}4+r_3\eta}N_3^{\frac{d-r}{2}-\frac{r_2}4-r_3\eta}\prod_{j=4}^{k+1}N_j^{\frac{d-r}{2}},  & \text{ if }k\geq 2,
\end{array}
\right.
$$
where $\eta$ can be any positive number, 
we have 
$$\left\|\prod_{j=1}^{k+1}\chi_{\lambda^j} f^j\right\|_{L^2(M)}\lesssim C(N_1,\ldots,N_{k+1})\prod_{j=1}^{k+1}\|f^j\|_{L^2(M)}.$$

\end{thm}

When $r=1$, the above theorem reduces to Theorem~\ref{mlspe}. But for $r\geq 2$, the above bound of the joint spectral projector is smaller than that of the spectral projector associated to the whole Laplace--Beltrami operator, and we use the former to bound joint eigenfunctions of all the individual Laplace--Beltrami operators on the sphere factors of the product manifold. The proof of the above theorem is a multi-parameter generalization of the argument of Burq--G\'erard--Tzvetkov, which in particular involves the H\"ormander parametrix for the half-wave operator on a compact manifold. The $k=1$ case was already treated by us in \cite{Zha21}. A simple but key idea of this paper is to treat the main and remainder terms of the parametrix at the same time at every step, instead of dealing with the remainder term at the end, which was the approach in \cite{BGT05, BGT052, Zha21} and would not easily generalize to higher $k$'s for a higher $r$. 

\begin{rem}
As observed in \cite{BGT052}, by testing against either highest weight or zonal spherical harmonics on spheres, the estimate in part (i) of Theorem~\ref{mlspe} is sharp except for the logarithmic factor, while the estimate in part (ii) is almost sharp---it would be sharp if one completely removes the arbitrarily small $\eta$ terms in the powers of $N_2$ and $N_3$. By considering products of eigenfunctions on each factor, a similar sharpness statement holds for Theorem~\ref{mljspe}. 
\end{rem}

Another key ingredient in proving Theorem~\ref{mls} is to get sharp estimates on some exponential sums which correspond to restricting the solution of the linear Schr\"odinger equation on tori to parts or none of the spatial dimensions while always keeping the time variable. We make the following: 

\begin{conj}[Anisotropic $L^\infty_{x_0} L^p_{x_1}$-type Strichartz estimate on $\mathbb{T}^{r_0}_{x_0}\times\mathbb{T}^{r_1}_{x_1}$]\label{rs}
Let $r_0\geq 1$ and $r_1\geq 0$ be integers and let $r=r_0+r_1$. 
Let ${\bf b}=(b_1,\ldots,b_r)\in\mathbb{R}^r$. For $N>1$, let 
$${\bf J}_{{\bf b},N}:=\{\xi=(n_1,\ldots,n_r)\in\mathbb{Z}^r: b_i\leq n_i\leq b_i+N, \ i=1,\ldots,r\}.$$
Let ${\xi}_1:=(n_{r_0+1},\ldots,n_{r})$, and ${x}_1\in\mathbb{R}^{r_1}$. 
Then for all $p>\max\{\frac{2(r_1+2)}{r},2\}$, it holds 
\begin{align}\label{rse}
    \left\|
    \sum_{\xi\in {\bf J}_{{\bf b},N}}a_\xi e^{-it|\xi|^2+i\langle x_1,{\xi}_1\rangle}
    \right\|_{L^p_{t,{x}_1}([0,2\pi]^{1+r_1})}\lesssim N^{\frac{r}{2}-\frac{r_1+2}{p}}\|a_\xi\|_{l^2(\mathbb{Z}^r)}
\end{align}
with the implicit constant independent of ${\bf b}$. 
\end{conj}

In particular, we will prove the following special case of the above conjecture in order to treat the cubic NLS posed on a product of two spheres both of dimension at least 4, at critical regularity.  

\begin{lem}\label{exp}
    Conjecture \ref{rs} holds for $r_1=0$. 
\end{lem}

\begin{rem}
    It is also natural to consider related noncompact models by replacing the toric factor $\mathbb{T}^{r_1}$ by a Euclidean one or more generally a waveguide of the form $\mathbb{R}^{r_2}\times\mathbb{T}^{r_1-r_2}$. We expect all results obtained in this paper to hold for such a mixed model $\mathbb{S}^{d_1}\times\cdots\times\mathbb{S}^{d_{r_0}}\times\mathbb{R}^{r_2}\times\mathbb{T}^{r_1-r_2}$ also, by following a similar framework. 
\end{rem}

\begin{rem}\label{rem: S3T}
    A specific model, namely the cubic NLS on $\mathbb{S}^3\times\mathbb{T}$, is particularly important due to its energy-critical nature. Although any subcritical regularity $s>s_c=1$ is already known to yield well-posedness (see Table \ref{lit}), local well-posedness in the energy space is still missing. Recently, the authors of \cite{DZZ25} successfully established the sharp bilinear eigenfunction estimate on $\mathbb{S}^3$, which immediately has the consequence of critical well-posedness for the cubic NLS on $\mathbb{S}^3\times\mathbb{T}^{r_1}$ with any $r_1\geq 2$. However, to get critical well-posedness for the case $r_1=1$, one would need the conjectured estimate \eqref{rse} for some $p<4$ in the case $(r_1,r_0)=(1,1)$, which is still missing. 
    On the related model $\mathbb{S}^3\times\mathbb{R}$, the authors of \cite{DZZ25} was able to prove critical well-posedness for the cubic NLS, by establishing a weaker substitute for the anisotropic $L^\infty_{x_0}L^p_{x_1}$-type Strichartz estimate on $\mathbb{T}_{x_0}\times\mathbb{R}_{x_1}$ (for any $p<4$)---how this could be adapted to the  $\mathbb{T}_{x_0}\times\mathbb{T}_{x_1}$ setting is under investigation. 
\end{rem}

\subsection*{Organization of paper} 
We first prove in Section \ref{mlj} the multi-linear joint spectral estimates of Theorem~\ref{mljspe}, and then use it to prove the multi-linear Strichartz estimates of Theorem~\ref{mls} in Section \ref{mlstr}. In Section \ref{wellposed}, we quickly prove Theorem~\ref{lwp} using Theorem~\ref{mls}, leaving well-known details to references. In Section \ref{expsum} we prove the exponential sum estimate of Lemma \ref{exp}.  In the Appendix, we review the known linear Strichartz estimates on products of spheres and tori and their consequences for local well-posedness, in particular making sure that Theorem~\ref{lwp} is all new.

\subsection*{Acknowledgments}
I would like to thank Ciprian Demeter for suggesting the method of proving Lemma \ref{exp} and various other helpful discussions. I would also like to thank Zehua Zhao for many helpful discussions. I thank the referee for their valuable comments, which improved the paper. 

\section{Proof of Theorem~\ref{mljspe}}\label{mlj}
\subsection{Review of Burq--G\'erard--Tzvetkov \texorpdfstring{\cite{BGT052}}{BGT052}}
Let $\chi\in C_c^\infty(\R)$ be a cutoff function such that $\widehat{\chi}(\tau)$ is supported in the set 
$\{\tau\in\mathbb{R}: \varepsilon\leq\tau\leq 2\varepsilon\}$, with $\varepsilon>0$ determined later in Lemma \ref{split}. Let $M$ be a compact Riemannian manifold of dimension $d\geq 2$ equipped with the Laplace--Beltrami operator $\Delta$. We will use: 

\begin{lem}[Lemma 2.3 of \cite{BGT052}]\label{param}
 Let $\lambda\geq 1$.     There exists $\varepsilon_0$ such that for all $\varepsilon\in(0,\varepsilon_0)$ and all $N\geq 1$, we have the splitting 
    \begin{align*}
        \chi(\sqrt{-\Delta}-\lambda)=T_\lambda f + R_\lambda f,
    \end{align*}
    with 
    \begin{align}\label{Rb}
        \|R_\lambda f\|_{H^k(M)}\lesssim_{k,N} \lambda^{k-N}\|f\|_{L^2(M)}, \ k=0,\ldots,N. 
    \end{align}
    Moreover, there exists $\delta>0$, and for every $x_0\in M$, a system of coordinates $V\in\mathbb{R}^d$ containing $0\in\mathbb{R}^d$ such that for $x\in V$, $|x|\leq \delta$, 
    \begin{align}\label{Tlambda}
        T_\lambda f(x)=\lambda^{\frac{d-1}{2}}\int_{\mathbb{R}^d} e^{it\varphi(x,y)}a(x,y,\lambda)f(y)\ dy,
    \end{align}
    where $a(x,y,\lambda)$ is a polynomial in $\lambda^{-1}$ with smooth coefficients supported in the set 
    \begin{align*}
        \{(x,y)\in V\times V: |x|\leq\delta \lesssim \varepsilon/C\leq|y|\leq C\varepsilon\},
    \end{align*}
    and $-\varphi(x,y)$ is the geodesic distance between $x$ and $y$. 
\end{lem}

We slightly changed the definition of $T_\lambda$ which would better serve our exposition. 
Next, we represent $y$ in geodesic coordinates as $y=\exp_0(r\omega)$, $\varepsilon/C<r< C\varepsilon$, $\omega\in\mathbb{S}^{d-1}$. 
Clearly there exists a smooth positive function $\kappa(r,\omega)$ such that $dy=\kappa(r,\omega)\ dr\ d\omega$. 
For $|x|\leq\delta$ and $\omega\in\mathbb{S}^{d-1}$, denote 
$$\varphi_r(x,\omega):=\varphi(x,\exp_0(r\omega)),$$
and $a_r(x,\omega,\lambda):=\kappa(r,\omega)a(x,\exp_0(r\omega),\lambda)$. 
Then \eqref{Tlambda} becomes
\begin{align}\label{Tformula}
T_\lambda f(x)=\lambda^{\frac{d-1}{2}}\int_{\varepsilon/C}^{C\varepsilon} 
\int_{\mathbb{S}^{d-1}}e^{i\lambda\varphi_r(x,\omega)}a_r(x,\omega,\lambda)f(\exp_0(r\omega))\ d\omega
\ dr.
\end{align}

Set 
$$
\Lambda(d,\lambda):=
\left\{
\begin{array}{ll}
\lambda^{\frac14}& \text{ if }d=2,\\
\lambda^{\frac12}\log^{\frac12}(\lambda)& \text{ if }d=3,\\
\lambda^{\frac{d-2}2}& \text{ if }d\geq 4.\\
\end{array}
\right.
$$
We will use:

\begin{lem}[Lemma 2.10 of \cite{BGT052}]\label{b1}
Let $x=(t,z)\in\mathbb{R}\times\mathbb{R}^{d-1}$ be any local system of coordinates near $(0,0)$. Then the operator 
$$f\in L^2(M)\mapsto T_\lambda f (t,z)\in L^2_tL^\infty_z(\mathbb{R}\times\mathbb{R}^{d-1})$$
is continuous with norm bounded by $\lesssim\Lambda(d,\lambda)$. 
\end{lem}

\begin{lem}[Lemma 2.11 of \cite{BGT052}]\label{b2}
    Let $d=2$ and $x=(t,z)\in\mathbb{R}\times\mathbb{R}^{d-1}$ be any local system of coordinates near $(0,0)$. Then the operator 
$$f\in L^2(M)\mapsto T_\lambda f(t,z)\in L^4_tL^\infty_z(\mathbb{R}\times\mathbb{R}^{d-1})$$
is continuous with norm bounded by $\lesssim \lambda^{\frac{1}{4}}$. 
\end{lem}

\begin{lem}[Lemma 2.12 of \cite{BGT052}]\label{b3}
    Let $d\geq 3$, $p>2$ and $x=(t,z)\in\mathbb{R}\times\mathbb{R}^{d-1}$ be any local system of coordinates near $(0,0)$. Then the operator 
$$f\in L^2(M)\mapsto T_\lambda f(t,z)\in L^p_tL^\infty_z(\mathbb{R}\times\mathbb{R}^{d-1})$$
is continuous with norm bounded by $\lesssim \lambda^{\frac{d-1}{2}-\frac{1}{p}}$. 
\end{lem}

\begin{lem}[Lemma 2.8, 2.9, 2.14 of \cite{BGT052}]\label{split}
Let $k\geq 0$ be a fixed integer. There exists a universal constant $\varrho>0$ depending only on $k$, such that for any $k+1$ spherical disks $W^1,\ldots,W^{k+1}$ all of radius $\varrho$ on the unit sphere $\mathbb{S}^{d-1}$, there exists a splitting $x=(t,z)\in\mathbb{R}\times\mathbb{R}^{d-1}$, $\varepsilon>0$, $\delta>0$, and a constant $c>0$, such that for $\varepsilon/C\leq r\leq C\varepsilon$, $|x|< \delta$, and $\omega\in\bigcup_{j=1}^{k+1} W^j$, the phase 
$\varphi_r(t,z,\omega)$ satisfies 
\begin{align}\label{nondeg}
    \left|\det_{i,j}\left(\frac{\partial^2\varphi_r(t,z,\omega)}{\partial z_j\partial\omega_i}\right)\right|\geq c.
\end{align}
Moreover, suppose the term $a_r(x,\omega,\lambda)$ in \eqref{Tformula} is supported in 
\begin{align*}
\{(x,\omega)\in\mathbb{R}^d\times\mathbb{S}^{d-1}: |x|< \delta, \ \omega\in \bigcup_{j=1}^{k+1} W^j\},
\end{align*}
then \eqref{nondeg} implies that the operator 
$$f\in L^2(M)\mapsto T_\lambda f(t,z)\in L^\infty_tL^2_z(\mathbb{R}\times\mathbb{R}^{d-1})$$
is continuous with norm bounded by $\lesssim 1$. 
\end{lem}
Note that in the statements of operator norm bounds in all the above four lemmas, as compared with \cite{BGT052}, we replaced the operator $T_\lambda^r$ defined by the inner integral in \eqref{Tformula} by the original operator $T_\lambda$, which would better serves our exposition. Another change in the statement of Lemma \ref{split} is the seemingly stronger requirement that the neighborhoods $W^j$ ($j=1,\ldots,k+1$) have a uniform diameter, which follows from the same proof of the Lemma 2.9 in \cite{BGT052}. This change would ease the following partition-of-unity argument.  

\subsection{Partition of unity}

Now we move to the setting of Theorem~\ref{mljspe} and consider the product manifold $M_1\times\cdots\times M_r$. 
Write for each $i=1,\ldots,r$, $j=1,\ldots,k+1$, 
$$\chi_i(\sqrt{-\Delta_i}-\lambda^j_i)=T_{i,\lambda^j_i}+R_{i,\lambda^j_i}$$
as in Lemma \ref{param}.

For each $i=1,\ldots,r$, we apply Lemma \ref{split} to $M_i$, which provides a $\varrho_i>0$. We pick a finite cover $\mathcal{W}=\{W_i\}$ of the unit sphere $\mathbb{S}^{d_i-1}$, where each $W_i$ is a spherical disk of radius $\varrho_i$. Now at each point of $M_i$, 
for each $(k+1)$-combination $W_i^1,\ldots,W_i^{k+1}$ from $\mathcal{W}$,  
Lemma \ref{split} provides choices of: 
\begin{itemize}
\item a splitting $x_i=(t_i,z_i)\in\mathbb{R}\times\mathbb{R}^{d_i-1}$;
\item an $\varepsilon_i>0$; and
\item a $\delta_i>0$.
\end{itemize} 
Though the splitting depends on the choice of the particular combination $W_i^1,\ldots,W_i^{k+1}$ of spherical disks, we can make the choices of $\varepsilon_i$ and $\delta_i$ {\it uniform} with respect to these combinations, as there are only finitely many $(k+1)$-combinations from the finite cover $\{W_i\}$.  
We pick our cutoff functions $\chi_i$ to have Fourier support in $\{\varepsilon_i\leq \tau\leq 2\varepsilon_i\}$\footnote{By the proof of Lemma 2.2 of \cite{BGT052}, to prove Theorem~\ref{mljspe}, it suffices to prove it with a particular choice of the cutoff functions $\chi_i$.}.
Then we use the neighborhoods $\{|x_i|<\delta_i\}$ to cover $M_i$, and pick a finite cover $\mathcal{V}=\{V_i\}$. 
Let $\{\rho_{i,V_i}(x_i)\}$ be a partition of unity subordinate to $\{V_i\}$. Also let $\{\phi_{i,W_i}(\omega_i)\}$ be a partition of unity subordinate to $\{W_i\}$. 

Now we partition each $T_{i,\lambda^j_i}$ as a finite sum of local operators of the same form as \eqref{Tformula}:
\begin{align}\label{sumT}
T_{i,\lambda^j_i} f^j_i(x_i)=\sum_{V^j_i\in\mathcal{V},\ W^j_i\in\mathcal{W}}T^{V^j_i,W^j_i}_{i,\lambda^j_i} f^j_i(x_i),
\end{align}
where 
\begin{align}\label{Tdetail}
T_{i,\lambda^j_i}^{V^j_i,W^j_i} f^j_i(x_i)=\lambda^{\frac{d_i-1}{2}}\int_{\varepsilon_i/C_i}^{C_i\varepsilon_i} 
\int_{\mathbb{S}^{d_i-1}}e^{i\lambda^j_i\varphi_{i,r}(x_i,\omega_i)}
a_{i,r}^{V^j_i,W^j_i}(x_i,\omega_i,\lambda^j_i)f^j_i(\exp_{i,0}(r\omega_i))\ d\omega_i
\ dr.
\end{align}
with 
\begin{align}\label{adetail}
a_{i,r}^{V^j_i,W^j_i}(x_i,\omega_i,\lambda^j_i)
=\rho_{i,V^j_i}(x_i)\phi_{i,W^j_i}(\omega_i)a_{i,r}(x_i,\omega_i,\lambda^j_i).
\end{align}
At the same time, for the remainder term $R_{i,\lambda^j_i} f_i(x_i)$ we also write:
\begin{align}\label{sumR}
    R_{i,\lambda^j_i} f^j_i(x_i)=\sum_{V^j_i\in\mathcal{V}}R_{i,\lambda^j_i}^{V^j_i} f^j_i(x_i),
\end{align}
where 
\begin{align}\label{Rdetail}
R_{i,\lambda^j_i}^{V^j_i} f^j_i(x_i)=\rho_{i,V^j_i}(x_i) R_{i,\lambda^j_i} f^j_i(x_i).
\end{align}
We have put the subscript ``$i$'' to each of the involved variables to refer to the $i$-th manifold $M_i$, and also the superscript ``$j$'' to refer to the $j$-th spectral projector as in the multi-linear bound.

\begin{lem}\label{Aij}
For $W_i^1,\ldots,W_i^{k+1}\in\mathcal{W}$, let $(t_i,z_i)$ be the chosen splitting that charts the neighborhood $V_i^j\in\mathcal{V}$. 
Let $A_{i,j}$ be either $T^{V^j_i,W^j_i}_{i,\lambda^j_i}$ or $R_{i,\lambda^j_i}^{V^j_i}$. Then:\\
(1) $\|A_{i,j}\|_{L^2(M_i)\to L^2_{t_i}L^\infty_{z_i}(\mathbb{R}\times\mathbb{R}^{d_i-1})}\lesssim \Lambda(d_i,\lambda^j_i)$; \\
(2) For $d_i=2$, $\|A_{i,j}\|_{L^2(M_i)\to L^4_{t_i}L^\infty_{z_i}(\mathbb{R}\times\mathbb{R}^{d_i-1})}\lesssim (\lambda^j_i)^\frac14$; \\
(3) For $d_i\geq 3$, $p>2$, $\|A_{i,j}\|_{L^2(M_i)\to L^p_{t_i}L^\infty_{z_i}(\mathbb{R}\times\mathbb{R}^{d_i-1})}\lesssim (\lambda^j_i)^{\frac{d_i-1}{2}-\frac1p}$; \\
(4) $\|A_{i,j}\|_{L^2(M_i)\to L^\infty_{t_i}L^2_{z_i}(\mathbb{R}\times\mathbb{R}^{d_i-1})}\lesssim 1$.
\end{lem}
\begin{proof}
For $A_{i,j}=T^{V^j_i,W^j_i}_{i,\lambda^j_i}$ given in \eqref{Tdetail}, we have that the function $a_{i,r}^{V_i^j,W_i^j}(x_i,\omega_i,\lambda_i^j)$ as defined in \eqref{adetail} has support in  
    $$\{(x_i,\omega_i)\in\mathbb{R}^d_i\times\mathbb{S}^{d_i-1}: |x_i|< \delta_i, \ \omega_i\in \bigcup_{j=1}^{k+1} W_i^j\},$$
so we can apply the operator bound in Lemma \ref{split} to yield (iv). Lemma \ref{b1}, \ref{b2} and \ref{b3} apply to $T^{V^j_i,W^j_i}_{i,\lambda^j_i}$ also, so to yield (i)-(iii). For $A_{i,j}=R_{i,\lambda^j_i}^{V^j_i}$ given in \eqref{Rdetail}, 
with \eqref{Rb} in hand, all the bounds follow from Sobolev embedding! 
\end{proof}

\subsection{Multi-linear estimates}
We consider 
$$\prod_{j=1}^{k+1}\chi_{\lambda^j} f^j(x_1,\ldots,x_r)=\prod_{j=1}^{k+1}\prod_{i=1}^r \chi_{i}(\sqrt{-\Delta_i}-\lambda_i^j)f^j(x_1,\ldots,x_r).$$ 
By \eqref{sumT} and \eqref{sumR}, the above is now a finite sum of terms of the form 
$$\prod_{j=1}^{k+1}\prod_{i=1}^r A_{i,j} f^j(x_1,\ldots,x_r),$$
where the operator $A_{i,j}$ is either $T^{V^j_i,W^j_i}_{i,\lambda^j_i}$ or $R^{V^j_i}_{i,\lambda^j_i}$. Let $(t_i,z_i)$ be the splitting associated to $W_i^1,\ldots,W_i^{k+1}$.

It suffices to get the same bound for 
$\left\|\prod_{j=1}^{k+1}\prod_{i=1}^r A_{i,j} f^j\right\|_{L^2(M_1\times\cdots\times M_r)}$. 
By H\"older's inequality, we have 
\begin{align*}
   \left\|\prod_{j=1}^{k+1}\prod_{i=1}^r A_{i,j} f^j\right\|_{L^2(M_1\times\cdots\times M_r)}
   =   \left\|\prod_{j=1}^{k+1}\prod_{i=1}^r A_{i,j} f^j\right\|_{L^2(V^j_1\times\cdots\times V^j_r)}
   \leq \prod_{j=1}^{k+1}\left\|\prod_{i=1}^r A_{i,j} f^j\right\|_{ L_{t_1}^{p_{1,j}}L_{z_1}^{q_{1,j}}\cdots L_{t_r}^{p_{r,j}}L_{z_r}^{q_{r,j}}}
\end{align*}
with $\sum_{j=1}^{k+1}1/p_{i,j}=\sum_{j=1}^{k+1}1/q_{i,j}=2$, $p_{i,j},q_{i,j}\geq 2$, $i=1,\ldots,r$, $j=1,\ldots,k+1$.  
Note the commutativity among the operators 
$A_{i,j}$ ($i=1,\ldots,r$ with $j$ fixed), and apply Minkowski's inequality, 
we have 
\begin{align*}
&\left\|\prod_{i=1}^r A_{i,j} f^j\right\|_{ L_{t_1}^{p_{1,j}}L_{z_1}^{q_{1,j}}\cdots L_{t_r}^{p_{r,j}}L_{z_r}^{q_{r,j}}}\\
&=\left\|\left\|A_{r,j}\prod_{i=1}^{r-1} A_{i,j} f^j\right\|_{ L_{t_r}^{p_{r,j}}L_{z_r}^{q_{r,j}}} \right\|_{ L_{t_1}^{p_{1,j}}L_{z_1}^{q_{1,j}}\cdots L_{t_{r-1}}^{p_{{r-1},j}}L_{z_{r-1}}^{q_{{r-1},j}}}\\
&\leq \|A_{r,j}\|_{L^2(M_r)\to L_{t_r}^{p_{r,j}}L_{z_r}^{q_{r,j}}}\left\|\left\|\prod_{i=1}^{r-1} A_{i,j} f^j\right\|_{L^2(M_r)} \right\|_{ L_{t_1}^{p_{1,j}}L_{z_1}^{q_{1,j}}\cdots L_{t_{r-1}}^{p_{{r-1},j}}L_{z_{r-1}}^{q_{{r-1},j}}}\\
&\leq \|A_{r,j}\|_{L^2(M_r)\to L_{t_r}^{p_{r,j}}L_{z_r}^{q_{r,j}}}\left\|\left\|\left\|\prod_{i=1}^{r-1} A_{i,j} f^j\right\|_{L_{t_{r-1}}^{p_{{r-1},j}}L_{z_{r-1}}^{q_{{r-1},j}}}\right\|_{L^2(M_r)} \right\|_{ L_{t_1}^{p_{1,j}}L_{z_1}^{q_{1,j}}\cdots L_{t_{r-2}}^{p_{{r-2},j}}L_{z_{r-2}}^{q_{{r-2},j}}}.
\end{align*}
Inductively, we arrive at 
\begin{align}\label{mpmi}
\left\|\prod_{i=1}^r A_{i,j} f^j\right\|_{ L_{t_1}^{p_{1,j}}L_{z_1}^{q_{1,j}}\cdots L_{t_r}^{p_{r,j}}L_{z_r}^{q_{r,j}}}
\leq \prod_{i=1}^r \|A_{i,j}\|_{L^2(M_i)\to L_{t_i}^{p_{i,j}}L_{z_i}^{q_{i,j}}}\|f^j\|_{L^2(M_1\times\cdots\times M_r)}.
\end{align}
Thus 
\begin{align*}
   \left\|\prod_{j=1}^{k+1}\prod_{i=1}^r A_{i,j} f^j\right\|_{L^2(M_1\times\cdots\times M_r)}
   \leq \left(\prod_{j=1}^{k+1}\prod_{i=1}^r \|A_{i,j}\|_{L^2(M_i)\to L_{t_i}^{p_{i,j}}L_{z_i}^{q_{i,j}}}\right)\cdot\left(\prod_{j=1}^{k+1}\|f^j\|_{L^2(M_1\times\cdots\times M_r)}\right).   
\end{align*}
Now assume $N_1\geq\cdots\geq N_{k+1}\geq 1$, where $N_j=\sqrt{(\lambda^j_1)^2+\cdots+(\lambda^j_r)^2}$. Of course $\lambda^j_i\leq N_j$ for all $i,j$. 
It suffices to evaluate the above bound numerically:\\
(i) For $k=1$, put $(p_{i,1},q_{i,1})=(\infty,2)$ and $(p_{i,2},q_{i,2})=(2,\infty)$, $i=1,\ldots,r$, and use (1) and (4) of Lemma \ref{Aij}, we get the desired bound in Theorem~\ref{mljspe}.\\
(ii) For $k\geq 2$, again first put $(p_{i,1},q_{i,1})=(\infty,2)$, $i=1,\ldots,r$. 
\begin{itemize}
    \item If $d_i=2$, put $(p_{i,2},q_{i,2})=(p_{i,3},q_{i,3})=(4,\infty)$, and $(p_{i,j},q_{i,j})=(\infty,\infty)$ for $j\geq 4$;
    \item If $d_i=3$, for any small number $\eta>0$, put $(p_{i,2},q_{i,2})=(1/(1/2-\eta),\infty)$,  $(p_{i,3},q_{i,3})=(1/\eta,\infty)$, 
    and $(p_{i,j},q_{i,j})=(\infty,\infty)$ for $j\geq 4$; 
    \item If $d_i=4$, put $(p_{i,2},q_{i,2})=(2,\infty)$, and $(p_{i,j},q_{i,j})=(\infty,\infty)$ for $j\geq 3$. 
\end{itemize}
Then using all the parts of Lemma \ref{Aij}, we also get the desired bound in Theorem~\ref{mljspe}. This finishes the proof of Theorem~\ref{mljspe}.

\section{Proof of Theorem~\ref{mls}}\label{mlstr}
We prove a more general version of Theorem~\ref{mls}, replacing the sphere factors by more general compact Riemannian manifolds $M$ which satisfy: 
\begin{enumerate}
\item[(A1)] \label{a1} \hypertarget{a1}{} There exist $a\in\mathbb{Q}_{<0}$ and $b\in\mathbb{Q}$, such that the eigenvalues of the Laplace--Beltrami operator $\Delta$ constitute (a subset of) $\{an(n+b): n\in\mathbb{Z}_{\geq 0}\}$;
\item[(A2)] \label{a2} \hypertarget{a2}{}   Suppose $-\Delta f=N^2 f$, $-\Delta g=M^2 g$, $-\Delta h=L^2 h$, and $N, M, L\geq 0$. There exists a constant $C>0$ depending only on the underlying manifold $M$, such that $fg$ is orthogonal to $h$ in $L^2(M)$ whenever $L\notin [N-CM,N+CM]$.
\end{enumerate}
Typical examples satisfying the above assumptions are rank-one compact symmetric spaces of compact type, namely, spheres, complex projective spaces, quaternionic projective spaces, octonionic projective spaces, and their finite quotients (see Chapter III \S9 of \cite{Hel08}). 

Let $M=M_1\times\cdots\times M_{r_0}\times\mathbb{T}^{r_1}$, where each $M_i$ satisfies the above two assumptions. 
For every $f\in L^2(M)$, we can write 
\begin{align}\label{sd}
f(x_0,x_1)=\sum_{\substack{\xi_0=(n_1,\ldots,n_{r_0})\in\mathbb{Z}_{\geq 0}^{r_0} \\ \xi_1=(n_{r_0+1},\ldots,n_{r})\in\mathbb{Z}^{r_1}} } f_{\xi}(x_0)e^{i\langle x_1, \xi_1\rangle},
\end{align}
where $x_0\in M_1\times\cdots\times M_{r_0}$, $x_1=(y_{r_0+1},\ldots,y_r)\in \mathbb{T}^{r_1}$, 
$\xi=(\xi_0,\xi_1)=(n_1,\ldots,n_r)\in \mathbb{Z}_{\geq 0}^{r_0}\times\mathbb{Z}^{r_1}$, 
and $f_{\xi}$ is a joint eigenfunction of the $\Delta_i$'s, $i=1,\ldots,r_0$. Without loss of generality, due to the nature of the multi-linear Strichartz estimate, we may simply assume that all eigenvalues of the Laplace--Beltrami operator $\Delta_i$ on $M_i$ belong to the set $\{-n^2: n\in\mathbb{Z}_{\geq 0}\}$. Then we can assume $-\Delta_i f_{\xi}=n_i^2 f_{\xi}$ for all $i=1,\ldots,r_0$. With $\Delta_i=\partial^2_{y_{i}}$ for $i=r_0+1,\ldots,r$, and $\Delta=\sum_{i=1}^{r}\Delta_i$, we have 
$$\Delta\left( f_{\xi}(x_0)e^{i\langle x_1, \xi_1\rangle}\right)=-|\xi|^2 \cdot f_{\xi}(x_0)e^{i\langle x_1, \xi_1\rangle},$$
so that 
\begin{align}\label{linearsum}
e^{it\Delta}f(x_0,x_1)=\sum_{\xi\in \mathbb{Z}_{\geq 0}^{r_0}\times\mathbb{Z}^{r_1} }e^{-it|\xi|^2} f_{\xi}(x_0)e^{i\langle x_1, \xi_1\rangle}.
\end{align}

Under the assumption of Theorem~\ref{mls}, following \cite{HTT11}, we first use \hyperlink{a2}{(A2)} to localize the spatial frequencies of $f^1$. Let $\mathcal{J}$ denote the collection of intervals of the form $[(m-1)N_2,mN_2)$, $m\in\mathbb{Z}$. For ${\bf J}=J_1\times\cdots\times J_r\in \mathcal{J}^r$, and for $f\in L^2(M)$ given in \eqref{sd}, let $P_{\bf J}$ be the spectral projector defined by 
$$P_{\bf J} f:= \sum_{\xi\in{\bf J}} f_{\xi}(x_0)e^{i\langle x_1, \xi_1\rangle}.$$ 
Given ${\bf J}=J_1\times\ldots\times J_r\in \mathcal{J}^r$ and ${\bf J}'=J'_1\times\ldots\times J'_r\in \mathcal{J}^r$, because of \hyperlink{a2}{(A2)}, there exists a constant $C>0$ depending only on the underlying manifolds $M_i$, such that whenever we have an $i\in\{1,\ldots,r\}$ with the property that the two intervals $J_i,J'_i$ are of a distance at least $C\cdot N_2$ away, then $(e^{it\Delta}P_{\bf J}f^1) (e^{it\Delta}f^2)\cdots (e^{it\Delta}f^{k+1})$ is orthogonal to $(e^{it\Delta}P_{\bf J'}f^1) (e^{it\Delta}f^2)\cdots (e^{it\Delta}f^{k+1})$ in $L^2(M_i)$ and thus also in $L^2([0,2\pi]\times M)$.
This reduces the desired inequalities in Theorem~\ref{mls} to those with $f^1$ replaced by a $P_{\bf J} f^1$. 

A second frequency localization is needed in order to treat the critical regularity. Let $\xi^0$ denote the center of the cube ${\bf J}$. Following \cite{HTT11}, let $M=\max\{{N_2^2}/{N_1},1\}$, and decompose ${\bf J}$ into slabs: 
${\bf J}=\bigcup_{m\in\mathbb{Z}} K_m$, where 
\begin{align}\label{slab}
K_m:=\{\xi\in {\bf J}: \langle\xi,\xi_0\rangle/|\xi_0|\in[(m-1)M,m M)\}.
\end{align}
For $f\in L^2(M)$ given by \eqref{sd}, consider the spectral projector defined by 
$$P_{K_m}f:= \sum_{\xi\in K_m} f_{\xi}(x_0)e^{i\langle x_1, \xi_1\rangle}$$ 
so that $P_{\bf J}=\sum_{m\in\mathbb{Z}} P_{K_m}$. 
Using the fact that $f^1$ is spectrally localized in $[N_1, 2N_1]$ with respect to $\sqrt{-\Delta}$, and ${\bf J}$ is a cube of side length $N_2$, a standard computation (as in the proof of Proposition 3.5 in \cite{HTT11}) shows that there exists a universal constant $C>0$ such that if $|m-m'|>C$, then $(e^{it\Delta}P_{ K_m}f^1) (e^{it\Delta}f^2)\cdots (e^{it\Delta}f^{k+1})$ is orthogonal to $(e^{it\Delta}P_{K_{m'}}f^1) (e^{it\Delta}f^2)\cdots (e^{it\Delta}f^{k+1})$ in $L^2_t([0,2\pi])$ and thus also in $L^2([0,2\pi]\times M)$. This further reduces the desired inequalities in Theorem~\ref{mls} to those with $f^1$ replaced by a $P_{K_m} f^1$. 

Write for each $j=1,\ldots,k+1$,
\begin{align*}
f^j(x_0,x_1)=\sum_{\xi^j=(\xi^j_0,\xi^j_1)\in\mathbb{Z}_{\geq 0}^{r_0}\times \mathbb{Z}^{r_1} } f^j_{\xi^j}(x_0)e^{i\langle x_1, \xi^j_1\rangle}.
\end{align*}
Then 
\begin{align*}
    (e^{it\Delta}P_{K_m}f^1\prod_{j=2}^{k+1}e^{it\Delta}f^j) (x_0,x_1)
    =\sum_{\substack{\xi^j\in \mathbb{Z}_{\geq 0}^{r_0}\times \mathbb{Z}^{r_1},\ j=1,\ldots,k+1 \\ \xi^1\in K_m}} e^{-it\sum_{j=1}^{k+1}|\xi^j|^2+i\langle x_1,\sum_{j=1}^{k+1}\xi_1^j\rangle}\prod_{j=1}^{k+1}f^j_{\xi^j}(x_0)
\end{align*}
First, we have
\begin{align*}
    \|e^{it\Delta}P_{K_m}f^1\prod_{j=2}^{k+1}e^{it\Delta}f^j\|_{L^2_{t,x_1}([0,2\pi]\times \mathbb{T}^{r_1})}^2
    =\sum_{l\in\mathbb{Z},\ \mu\in \mathbb{Z}^{r_1}}\left|\sum_{\substack{ \sum_{j=1}^{k+1}|\xi^j|^2= l \\ \sum_{j=1}^{k+1}\xi_1^j=\mu}}\prod_{j=1}^{k+1}f^j_{\xi^j}(x_0)\right|^2.
\end{align*}
Then by Minkowski's inequality in $L^2$, we have 
\begin{align*}
    \|e^{it\Delta}P_{K_m}f^1\prod_{j=2}^{k+1}e^{it\Delta}f^j\|_{L^2([0,2\pi]\times M)}^2
    \leq \sum_{l\in\mathbb{Z},\ \mu\in \mathbb{Z}^{r_1}} \left(\sum_{\substack{ \sum_{j=1}^{k+1}|\xi^j|^2= l \\ \sum_{j=1}^{k+1}\xi_1^j=\mu}}\left\|\prod_{j=1}^{k+1}f^j_{\xi^j}(x_0)\right\|_{L^2(M_1\times \cdots\times M_{r_0})}\right)^2.
\end{align*}
As each $f^j_{\xi^j}$ is a joint eigenfunction of the Laplace--Beltrami operators $\Delta_i$ ($i=1,\ldots,r_0$) with the joint spectrum $\xi^j_0=(n^j_1,\ldots,n^j_{r_0})$, we have $$\chi_{\xi^j_0}f^j_{\xi^j}=f^j_{\xi^j},$$
where $\chi_{\xi^j_0}$ is a joint spectral projector for the product manifold $M_1\times\cdots\times M_{r_0}$ as defined in Theorem~\ref{mljspe}. Apply Theorem~\ref{mljspe}, noting that $|\xi^j_0|\leq|\xi^j|=N_j$, $j=1,\ldots,k+1$, 
we have 
$$\left\|\prod_{j=1}^{k+1}f^j_{\xi^j}\right\|_{L^2(M_0)}\lesssim C(N_1,\ldots,N_{k+1})\prod_{j=1}^{k+1}\|f^j_{\xi^j}\|_{L^2(M_0)},$$
where 
\begin{align}\label{CM0}
C(N_1,\ldots,N_{k+1})
=\left\{
\begin{array}{ll}
   N_2^{\frac{d_0-2r_0}{2}+\frac{r_2}4}(\log N_2)^{\frac{r_3}2},  & \text{ if }k=1, \\
   N_2^{\frac{d_0-2r_0}{2}+\frac{r_2}4+r_3\eta}N_3^{\frac{d_0-r_0}{2}-\frac{r_2}4-r_3\eta}\prod_{j=4}^{k+1}N_j^{\frac{d_0-r_0}{2}},  & \text{ if }k\geq 2,
\end{array}
\right.
\end{align}
where $\eta$ is any positive number. Here we used $d_0$ to denote the dimension of $M_0$. 
Plug the above inequality into the further above one, we get 
\begin{align*}
    &\|e^{it\Delta}P_{K_m}f^1\prod_{j=2}^{k+1}e^{it\Delta}f^j\|_{L^2([0,2\pi]\times M)}
    \lesssim C(N_1,\ldots,N_{k+1}) \left(\sum_{l\in\mathbb{Z},\ \mu\in \mathbb{Z}^{r_1}} 
    \left(\sum_{\substack{ \sum_{j=1}^{k+1}|\xi^j|^2= l \\ \sum_{j=1}^{k+1}\xi_1^j=\mu}}\prod_{j=1}^{k+1}
    \|f^j_{\xi^j}\|_{L^2(M_0)}\right)^2\right)^{\frac12}\\
&\lesssim C(N_1,\ldots,N_{k+1})
\left\|
\sum_{\xi^1\in K_m} e^{-it|\xi^1|^2+i\langle x_1, \xi^1_1\rangle}\|f^1_{\xi^1}\|_{L^2(M_0)}\prod_{j=2}^{k+1}\sum_{|\xi^j|\in[N_j,2N_j]}e^{-it|\xi^j|^2+i\langle x_1, \xi^j_1\rangle} \|f^j_{\xi^j}\|_{L^2(M_0)}
\right\|_{L^2_{t,x_1}} \\
&\lesssim C(N_1,\ldots,N_{k+1})
\left\|
\sum_{\xi^1\in K_m} e^{-it|\xi^1|^2+i\langle x_1, \xi^1_1\rangle}\|f^1_{\xi^1}\|_{L^2(M_0)}\right\|_{L^{p_1} _{t,x_1}}\prod_{j=2}^{k+1} \left\| \sum_{|\xi^j|\in[N_j,2N_j]}e^{-it|\xi^j|^2+i\langle x_1, \xi^j_1\rangle} \|f^j_{\xi^j}\|_{L^2(M_0)}
\right\|_{L^{p_j} _{t,x_1}}  
\end{align*}
where we used H\"older's inequality with $p_j\geq 2$, $\sum_{j=1}^{k+1}1/p_j=1/2$. Now:\\
(1) For $k=1$, $r\geq 3$ or $r=r_0=2$, we put $p_1=p_2=4$, and use both \eqref{i1} and \eqref{i2} of Lemma \ref{e1} below to get: for some $\delta_0>0$, 
$$\|e^{it\Delta}P_{K_m}f^1 \cdot e^{it\Delta}f^2\|_{L^2([0,2\pi]\times M)}\lesssim C(N_1,N_2)\left(\frac{N_2}{N_1}+\frac{1}{N_2}\right)^{\delta_0}N_2^{r-\frac{2+r_1}{2}}\|f^1\|_{L^2(M)}\|f^2\|_{L^2(M)}.$$
(2) For $k=1$, $r=2$, $r_0=0,1$, we also put $p_1=p_2=4$, and use \eqref{i1} of Lemma \ref{e1} to get: for all $\varepsilon>0$, 
$$\|e^{it\Delta}P_{K_m}f^1 \cdot e^{it\Delta}f^2\|_{L^2([0,2\pi]\times M)}\lesssim C(N_1,N_2)N_2^{r-\frac{2+r_1}{2}+\varepsilon}\|f^1\|_{L^2(M)}\|f^2\|_{L^2(M)}.$$
(3) For $k\geq 2$, $r\geq 3$, we put $p_1=p_2=4$, $p_j=\infty$ for $j\geq 3$, and use both \eqref{i1} and \eqref{i2} of  Lemma \ref{e1} to get: there exists  $\delta_0>0$ such that for all $\delta\in[0,\delta_0)$,

\begin{align*}
&\|e^{it\Delta}P_{K_m}f^1\prod_{j=2}^{k+1}e^{it\Delta}f^j\|_{L^2([0,2\pi]\times M)}\\
    &\lesssim C(N_1,\ldots,N_{k+1})
    \left(\frac{N_2}{N_1}+\frac{1}{N_2}\right)^{\delta_0}N_2^{r-\frac{2+r_1}{2}}\prod_{j=3}^{k+1}N_j^{\frac{r}{2}}\prod_{j=1}^{k+1}\|f^j\|_{L^2(M)}\\
    &\lesssim C(N_1,\ldots,N_{k+1})
    \left(\frac{N_{k+1}}{N_1}+\frac{1}{N_2}\right)^{\delta}N_2^{r-\frac{2+r_1}{2}+\delta(k-1)}\prod_{j=3}^{k+1}N_j^{\frac{r}{2}-\delta}\prod_{j=1}^{k+1}\|f^j\|_{L^2(M)}.
\end{align*}
(4) For $k\geq 2$, $r=2$, for any small $\varepsilon>0$, we put $p_3=(2+r_1)/\varepsilon$ and $p_j=\infty$ for all $j\geq 4$, and $p_1=p_2=2/(\frac12-\frac{\varepsilon}{2+r_1})>4$. Using both \eqref{i1} and \eqref{i2} of Lemma \ref{e1}, there exists $\delta_0>0$, such that for all $\delta\in[0,\delta_0)$, 
\begin{align*}
&\|e^{it\Delta}P_{K_m}f^1\prod_{j=2}^{k+1}e^{it\Delta}f^j\|_{L^2([0,2\pi]\times M)}\\
   & \lesssim C(N_1,\ldots,N_{k+1})
    \left(\frac{N_2}{N_1}+\frac{1}{N_2}\right)^{\delta_0}N_2^{r-\frac{2+r_1}{2}+\varepsilon}N_3^{\frac{r}{2}-\varepsilon} \prod_{j=4}^{k+1}N_j^{\frac{r}{2}}\prod_{j=1}^{k+1}\|f^j\|_{L^2(M)}\\
    &\lesssim C(N_1,\ldots,N_{k+1})
    \left(\frac{N_{k+1}}{N_1}+\frac{1}{N_2}\right)^{\delta}N_2^{r-\frac{2+r_1}{2}+\varepsilon+\delta(k-1)}N_3^{\frac{r}{2}-\varepsilon-\delta} \prod_{j=4}^{k+1}N_j^{\frac{r}{2}-\delta}\prod_{j=1}^{k+1}\|f^j\|_{L^2(M)}.
\end{align*}
Combined with \eqref{CM0}, we see that (1) \& (2) of the above yield (i), (ii) and (iii) of Theorem~\ref{mls}, (3) yields (iv), and (4) yields (v).

\begin{lem}\label{e1}
Let $r_0\geq 1$ and $r_1\geq 0$ be integers and let $r=r_0+r_1$. let 
$$p_0=\left\{\begin{array}{ll} 
2, &\text{ if }r=r_0\geq 2, \\
\frac{2(r+2)}{r}, & \text{ otherwise}.
\end{array}\right.$$
Assume $p>p_0$. \\
(i) For $N\geq 1$, ${\bf b}\in\mathbb{R}^r$, let ${\bf J}_{{\bf b},N}={\bf b}+[0,N]^r$. 
Then
\begin{align}\label{i1}
\left\|
\sum_{\xi\in {\bf J}_{{\bf b},N}} e^{-it|\xi|^2+i\langle x_1, \xi_1\rangle}a_{\xi}\right\|_{L^{p} _{t,x_1}([0,2\pi]^{1+r_1})}
\lesssim N^{\frac{r}{2}-\frac{2+r_1}{p}}\|a_{\xi}\|_{l^2_{\xi}},
\end{align}
uniformly in ${\bf b}$. \\
(ii) For $N_1\geq N_2\geq 1$, ${\bf b}\in\mathbb{R}^r$, $m\in\mathbb{Z}$, let ${\bf J}={\bf b}+[0,N_2]^r$, and let $K_m$ be a slab in ${\bf J}$ as defined in \eqref{slab}. Then there exists $\delta_0=\delta_0(p)>0$ such that
\begin{align}\label{i2}
\left\|
\sum_{\xi\in K_m} e^{-it|\xi|^2+i\langle x_1, \xi_1\rangle}a_{\xi}\right\|_{L^{p}_{t,x_1}([0,2\pi]^{1+r_1})}
\lesssim \left(\frac{N_2}{N_1}+\frac{1}{N_2}\right)^{\delta_0}N_2^{\frac{r}{2}-\frac{2+r_1}{p}}\|a_{\xi}\|_{l^2_{\xi}},
\end{align}
uniformly in ${\bf b}$ and $m$. \\
\end{lem}
\begin{proof}
\eqref{i1} is the same as \eqref{rse} of Conjecture \ref{rs}. We first establish \eqref{i1} for all $p>\frac{2(r+2)}2$. 
Denote $x=(x_0,x_1)\in \mathbb{R}^{r_0}\times\mathbb{R}^{r_1}$. Without loss of generality, we may assume ${\bf b}\in\mathbb{Z}^r$. We have 
\begin{align*}
    \left\|
    \sum_{\xi\in {\bf J}_{{\bf b},N}}a_\xi e^{-it|\xi|^2+i\langle x_1,{\xi}_1\rangle}
    \right\|_{L^p_{t,x_1}([0,2\pi]^{1+r_1})}
    &\leq     \left\|
    \sum_{\xi\in {\bf J}_{{\bf b},N}}a_\xi e^{-it|\xi|^2+i\langle x, \xi\rangle}
    \right\|_{L^p_{t,x_1}L^\infty_{x_0}([0,2\pi]^{1+r})}\\
    &\leq     \left\|
    \sum_{\nu\in [0,N]^{1+r}}a_\xi e^{-it|\nu|^2+i\langle x-2t{\bf b}, \nu\rangle}
    \right\|_{L^p_{t,x_1}L^\infty_{x_0}([0,2\pi]^{1+r})},  
    \end{align*}
where we used the change of variables $\nu=\xi-{\bf b}$.    
    Use Bernstein's inequality on $\mathbb{T}^{r_0}$, then 
\begin{align*}
    \left\|
    \sum_{\xi\in {\bf J}_{{\bf b},N}}a_\xi e^{-it|\xi|^2+i\langle x_1,{\xi}_1\rangle}
    \right\|_{L^p_{t,x_1}([0,2\pi]^{1+r_1})}
    &\lesssim N^{\frac{r_0}{p}} \left\|
    \sum_{\nu\in [0,N]^{1+r}}a_\xi e^{-it|\nu|^2+i\langle x-2t{\bf b}, \nu\rangle}
    \right\|_{L^p_{t,x}([0,2\pi]^{1+r})}.
\end{align*}
The right hand side of the above is the same as 
$$N^{\frac{r_0}{p}} \left\|
    \sum_{\nu\in [0,N]^{1+r}}a_\xi e^{-it|\nu|^2+i\langle x, \nu\rangle}
    \right\|_{L^p_{t,x}([0,2\pi]^{1+r})},$$
which is bounded by $N^{\frac{r_0}{p}}\cdot N^{\frac{r}2-\frac{r+2}{p}}\|a_\xi\|_{l^2({\bf J}_{{\bf b},N})}=N^{\frac{r}2-\frac{r_1+2}{p}}\|a_\xi\|_{l^2({\bf J}_{{\bf b},N})}$ for all $p>\frac{2(r+2)}{r}$, using the Strichartz estimate for the Schr\"odinger equation on tori \cite{BD15}. This finishes the proof of \eqref{i1} for $p>\frac{2(r+2)}{r}$; the other case for $r=r_0\geq 2$ is a consequence of Lemma \ref{exp} which we prove later. Next we have 
\begin{align*}
\left\|
\sum_{\xi\in K_m} e^{-it|\xi|^2+i\langle x_1, \xi_1\rangle}a_{\xi}\right\|_{L^{\infty} _{t,x_1}}
\lesssim |K_m|^{\frac12}\|a_{\xi}\|_{l^2_{\xi}}\lesssim M^{\frac12}N_2^{\frac{r-1}{2}}\|a_{\xi}\|_{l^2_{\xi}},
\end{align*}
where $M=\max\{N_2^2/N_1,1\}$. 
Recall that $K_m$ is a subset of a cube ${\bf J}$ of side length $N_2$, so we can interpolate the above bound and \eqref{i1} with $N=N_2$, which yields \eqref{i2} for all $p>p_0$. 

\end{proof}

\section{Proof of Theorem~\ref{lwp}}\label{wellposed}

Checking the exponents of the spectral parameters $N_j$ in the multi-linear Strichartz estimates of Theorem~\ref{mls} and keeping in mind the order relation 
$N_1\geq \cdots\geq N_{k+1}\geq 1$, we see that: 
\begin{itemize}
    \item  For all the cases in (1) of Theorem~\ref{lwp}, the multi-linear Strichartz estimates in Theorem~\ref{mls} imply the following one with uniform exponents: there exists $\delta>0$, such that 
\begin{align*}
\left\|\prod_{j=1}^{k+1}e^{it\Delta}f^j\right\|_{L^2(I\times M)}
&\lesssim 
\left(\frac{N_{k+1}}{N_1}+\frac1{N_2}\right)^{\delta}
\prod_{j=2}^{k+1}
N_j^{\frac{d}{2}-\frac{1}{k}}\prod_{j=1}^{k+1}\|f^j\|_{L^2(M)}.
\end{align*}
The above also holds for the extra cases $k=1$, $r\geq 3$, $r_2=r_3=0$, which however were already treated in \cite{Zha21}. 
Using the now standard theory of $U^p$ and $V^p$ spaces introduced by Koch--Tataru \cite{KT05}, the above multi-linear Strichartz estimate implies local well-posedness of the NLS for initial data in $H^{s}(M)$, for all $s\geq s_c=\frac{d}{2}-\frac1{k}$. We refer to the proof of Theorem~1.1 of \cite{HS15} for a detailed derivation, and how it does not depend on the specifics of the underlying manifold $M$. 

\item For all the cases in (2) of Theorem~\ref{lwp}, the multi-linear Strichartz estimates in Theorem~\ref{mls} implies 
\begin{align*}
\left\|\prod_{j=1}^{k+1}e^{it\Delta}f^j\right\|_{L^2(I\times M)}
&\lesssim 
\prod_{j=2}^{k+1}
N_j^{\frac{d}{2}-\frac{1}{k}+\varepsilon}\prod_{j=1}^{k+1}\|f^j\|_{L^2(M)}
\end{align*}
for all $\varepsilon>0$. 
Using the now standard theory of Fourier restriction spaces introduced by Bourgain \cite{Bou93}, the above multi-linear Strichartz estimate implies local well-posedness of the NLS for initial data in $H^{s}(M)$, for all $s> s_c=\frac{d}{2}-\frac1{k}$. We refer to the proof of Theorem~3 of \cite{BGT05} for a detailed derivation, and how it does not depend on the specifics of the underlying manifold $M$. The above estimate also holds for the extra cases: (1) $r_2=0$, $k=1$, which was already treated in \cite{Zha21}; 
(2) $r_2=2$, $r\geq 4$, $k= 2$, 
and $r_2=3$, $k= 4$, but for these cases almost critical local well-posedness also follows from the approach by linear Strichartz estimates---see Theorem~\ref{llwp}.


\item For all the cases in (3) of Theorem~\ref{lwp}, the multi-linear Strichartz estimates in Theorem~\ref{mls} implies 
\begin{align*}
\left\|\prod_{j=1}^{k+1}e^{it\Delta}f^j\right\|_{L^2(I\times M)}
&\lesssim 
\prod_{j=2}^{k+1}
N_j^{s_0+\varepsilon}\prod_{j=1}^{k+1}\|f^j\|_{L^2(M)}
\end{align*}
for all $\varepsilon>0$. Again using the Fourier restriction spaces of Bourgain, the above multi-linear Strichartz estimate implies local well-posedness of the NLS for initial data in $H^{s}(M)$, for all $s> s_0$. The above estimate also holds for the extra cases: 
$r_2=2,3$, $k=1$, $s_0=\frac{d}2-1+\frac{r_2}4$, and $r_2=3$, $k= 2,3$, $s_0=\frac{d}2-\frac{1}4$, but for these cases the corresponding sub-critical ($s>s_0$) local well-posedness again also follows from the approach by linear Strichartz estimates---see Theorem~\ref{llwp}.   

\end{itemize}

\section{Proof of Lemma \ref{exp}}\label{expsum}
The $r=1$ case of Lemma \ref{exp} 
follows from Lemma 3.1 of \cite{Her13}, which is a generalization of Bourgain's result of the ${\bf b}=0$ case \cite{Bou89}. For $r\geq 2$,
the goal is to prove that for all $p>2$, it holds 
\begin{align*}
    \left\|
    \sum_{\xi\in {\bf J}_{{\bf b},N}}a_\xi e^{-it|\xi|^2}
    \right\|_{L^p_{t}([0,2\pi])}\lesssim N^{\frac{r}{2}-\frac{2}{p}}\|a_\xi\|_{l^2(\mathbb{Z}^r)}.
\end{align*}
We first establish that for all $\varepsilon>0$, 
\begin{align}\label{2loss}
    \left\|
    \sum_{\xi\in {\bf J}_{{\bf b},N}}a_\xi e^{-it|\xi|^2}
    \right\|_{L^2_{t}([0,2\pi])}\lesssim N^{\frac{r}{2}-1+\varepsilon}\|a_\xi\|_{l^2(\mathbb{Z}^r)}.
\end{align}
We have 
\begin{align}\label{counting}
    &\left\|
    \sum_{\xi\in {\bf J}_{{\bf b},N}}a_\xi e^{-it|\xi|^2}
    \right\|_{L^2_{t}([0,2\pi])} 
    = \left(\sum_{A}\left|\sum_{|\xi|^2=A, \ \xi\in {\bf J}_{{\bf b},N}}a_{\xi}\right|^{2}\right)^{\frac12} \nonumber \\
    &\lesssim \sup_A\#\{\xi=(n_1,\ldots,n_r)\in {\bf J}_{{\bf b},N}: |\xi|^2=A\}^{\frac12}
  \|a_\xi\|_{l^2(\mathbb{Z}^r)} \nonumber \\
   &\lesssim N^{\frac{r}{2}-1}\sup_{A'}\#\{(n_1,n_2)\in[b_1,b_1+N]\times[b_2,b_2+N] \cap \mathbb{Z}^2: n_1^2+n_2^2=A'\}^{\frac12}
   \|a_\xi\|_{l^2(\mathbb{Z}^r)}.
\end{align}
Let $S$ denote the set $\{(n_1,n_2)\in[b_1,b_1+N]\times[b_2,b_2+N] \cap \mathbb{Z}^2: n_1^2+n_2^2=A'\}$. It suffices to show $\#S\lesssim N^\varepsilon$. 
For $A'\leq N^{8}$, this is a consequence of the standard divisor bound. 
For $A'>N^{8}$, we can use geometry to conclude that $S$ has at most 2 elements as follows. We are counting the number of lattice points on a circle of radius $\sqrt{A'}>N^4$ inside a square of side length $N$. Suppose there are two distinct points $A,B$ in $S$. Let $l$ denote the line containing $A$ and $B$. Suppose $C$ is another point in $S$. First, as the circle is curved, $C$ cannot be in $l$, and thus the distance $d(C,l)$ between $C$ and $l$ satisfies $d(C,l)\gtrsim N^{-1}$ as $A,B,C$ are all lattice points lying in a square of side length $N$. This further implies that the inscribed angle $\angle{CAB}\gtrsim N^{-2}$. But we also have $\angle{CAB}=|\overset{\frown}{CB}|/(2\sqrt{A'})\lesssim N/N^4=N^{-3}$, which yields a contradiction. We have finished the proof of \eqref{2loss}.

Now we follow 
Herr's proof of the $r=1$ case in \cite{Her13}, as well as Bourgain's original argument in \cite{Bou89}, and also Section 2 of \cite{Dem25}. Let $p>2$ and $\varepsilon>0$. For the sake of exposition and without loss of generality, we may assume ${\bf b}=(b_1,\ldots,b_r)\in\mathbb{Z}^n$ and $N\in\mathbb{Z}_{>0}$. 
Pick a sequence $\sigma$ satisfying:

(i) For all $n \in \mathbb{Z}$, $0\leq \sigma(n)\leq 1$. For all $n\in[0,N]$, $\sigma(n)=1$. For all $n$ such that $n<-N$ or $n>2N$, $\sigma(n)=0$.

(ii) The sequence $\sigma(n+1)-\sigma(n)$ is bounded by $N^{-1}$ and has variation bounded
by  $N^{-1}$. 

\noindent Then define $\sigma_{b_i}(n):=\sigma(n-b_i)$, $i=1,\ldots,r$. Let $\sigma_{\bf b}(\xi)=\sigma_{\bf b}(n_1,\ldots,n_r):=\prod_{i=1}^r \sigma_{b_i}(n_i)$. Fix $\varepsilon>0$. 
Assume $\|a_\xi\|_{l^2(\mathbb{Z}^r)}=1$. It suffices to prove the distributional inequality 
\begin{align}\label{dis}
\sup_{{\bf b}}\left|\left\{t\in[0,1]: \left|\sum_{\xi\in\mathbb{Z}^r}\sigma_{\bf b}(\xi)a_\xi e^{-2\pi it|\xi|^2}\right|>\delta N^{\frac r2}\right\}\right|\lesssim N^{-2}\delta^{-2-\varepsilon}
\end{align}
for all $0<\delta<1$. Note that here we scaled the time interval from $[0,2\pi]$ to $[0,1]$ for the sake of exposition. 

It follows from \eqref{2loss} that 
\begin{align*}
\sup_{{\bf b}}\left|\left\{t\in[0,1]: \left|\sum_{\xi\in\mathbb{Z}^r}\sigma_{\bf b}(\xi)a_\xi e^{-2\pi it|\xi|^2}\right|>\delta N^{\frac r2}\right\}\right|\lesssim N^{-2+\frac{\varepsilon}{100}}\delta^{-2}. 
\end{align*}
Thus it suffices to prove \eqref{dis} for $\delta>N^{-\frac1{100}}$. 

Let 
$$f_{b_i}(t):=\sum_{n\in\mathbb{Z}} \sigma^2_{b_i}(n) e^{2\pi itn^2}.$$ 
Note that the sequence $\sigma^2$ still satisfies the above two conditions (i) and (ii). Then Weyl differencing provides
\begin{align}\label{major}
    \left| f_{b_i}(t)\right|\lesssim q^{-1/2}(|t-a/q|+N^{-2})^{-1/2},
\end{align}
for any $1\leq a\leq q<N$, gcd($a,q$)=1 and $|t-a/q|<(qN)^{-1}$, uniformly in $b_i\in\mathbb{Z}$. The above is (32) of \cite{Her13}.
We define the major arcs $\mathcal{M}$ to be the disjoint union of the sets 
$\mathcal{M}(a,q)=\{t\in[0,1]: |t-a/q|\leq N^{\frac1{10}-2}\}$, for any $1\leq a\leq q\leq N^\frac1{10}$, gcd($a,q$)=1. If $t\in[0,1]\setminus \mathcal{M}$, then an application of Dirichlet's approximation theorem gives
\begin{align}\label{minor}
    \left| f_{b_i}(t)\right|\lesssim N^{1-\frac1{20}}.
\end{align}
The above is (33) of \cite{Her13}. 

To prove \eqref{dis}, it suffices to bound the number $R$ of $N^{-2}$ separated points $t_1,\ldots,t_R\in[0,1]$ such that 
$$\left|\sum_{\xi\in \mathbb{Z}^r}\sigma_{\bf b}(\xi)a_\xi e^{-2\pi it_r|\xi|^2}\right|>\delta N^{\frac r2}, \text{ for all }r=1,\ldots,R.$$
Let $|c_r|=1$ such that 
$$\left|\sum_{\xi\in \mathbb{Z}^r}\sigma_{\bf b}(\xi)a_\xi e^{-2\pi it_r|\xi|^2}\right|=c_r \sum_{\xi\in \mathbb{Z}^r}\sigma_{\bf b}(\xi)a_\xi e^{-2\pi it_r|\xi|^2}. $$
Using the Cauchy--Schwarz and triangle inequalities we find 
\begin{align*}
    R\delta N^{\frac r2}
   & \leq \sum_{r=1}^R c_r \sum_{\xi\in \mathbb{Z}^r}\sigma_{\bf b}(\xi)a_\xi e^{-2\pi it_r|\xi|^2}=\sum_{\xi\in \mathbb{Z}^r}a_\xi \sum_{r=1}^R c_r  \sigma_{\bf b}(\xi)e^{-2\pi it_r|\xi|^2}\\
   &\leq \left(\sum_{\xi\in \mathbb{Z}^r}\left|\sum_{r=1}^R c_r  \sigma_{\bf b}(\xi)e^{-2\pi it_r|\xi|^2} \right|^2\right)^{\frac12}\\
   &\leq\left(\sum_{1\leq r,r'\leq R} \left| \sum_{\xi\in \mathbb{Z}^r}\sigma_{\bf b}^2(\xi)e^{2\pi i(t_r-t_{r'})|\xi|^2} \right|\right)^{\frac12}
   =\left(\sum_{1\leq r,r'\leq R} \left| 
   \prod_{i=1}^r f_{b_i}(t_r-t_{r'})\right|\right)^{\frac12}.
\end{align*}
Let $\gamma\geq 1$ be fixed such that $r\gamma>2$. That is, for $r=2$, we can pick any $\gamma>1$, while for $r\geq 3$, we may put $\gamma=1$. Then the above estimate combined with H\"older's inequality yields 
\begin{align}\label{meat}
\sum_{1\leq r,r'\leq R} \left| 
   \prod_{i=1}^r f_{b_i}(t_r-t_{r'})\right|^\gamma\gtrsim R^{2}\delta^{2\gamma} N^{r\gamma}.
\end{align}
Define  
$F(\theta)=(N^2|\sin\theta|+1)^{-r\gamma/2}$,
and let 
$$G(t)=\sum_{q\leq Q,\ 1\leq a\leq q, \ (a,q)=1}q^{-r\gamma/2}F(t-a/q),$$
which will provide an upper bound for \eqref{meat}. The condition $r\gamma>2$ makes sure that 
$\|F\|_{L^1([0,2\pi])}\lesssim N^{-2}$. Here we pick $Q=C\delta^{-5}\leq C N^{\frac1{20}}$, where $C$ is a large but absolute constant independent of $N$ and $\delta$. This choice of $Q$ makes sure that:
\begin{itemize}
    \item The involved fractions $a/q$ in the above sum belong to the major arcs $\mathcal{M}$ defined above, so that when $t_r-t_{r'}$ lies in one of these major arcs, we can use \eqref{major} to bound the left hand side of \eqref{meat};
    \item When $t_r-t_{r'}$ lies in some $\mathcal{M}(a,q)\subset\mathcal{M}$ but with $q>Q$, then the \eqref{major} is still applicable so that  $\sum_{r,r'}\prod_{i=1}^r|f_{b_i}(t_r-t_{r'})|^\gamma\lesssim R^2Q^{-r\gamma/2}N^{r\gamma}=  C^{-r\gamma/2}R^2\delta^{5r\gamma/2}N^{r\gamma}$, which is negligible when compared with the right hand side of \eqref{meat};
    \item When $t_r-t_{r'}$ lies outside of $\mathcal{M}$, then \eqref{minor} implies that 
    $\sum_{r,r'}\prod_{i=1}^r|f_{b_i}(t_r-t_{r'})|^\gamma\lesssim R^2N^{r\gamma(1-1/20)}$, and since $\delta>N^{-1/100}$, this is also negligible when compared with the right hand side of \eqref{meat}.
\end{itemize}
The above points together yield
$$\sum_{1\leq r,r'\leq R}G(t_r-t_{r'})\gtrsim R^2\delta^{2\gamma}.$$
Following the remaining arguments in \cite[pp.~306-307]{Bou89} verbatim, we arrive at $R\lesssim \delta^{-2-\varepsilon}$, which implies \eqref{dis}. 

\begin{rem}\label{remcounting}
For all $r_0\geq 2$, a counting estimate similar to \eqref{counting} also works, which yields \eqref{rse} with an $N^\varepsilon$ loss. It is possible to remove this loss following a similar approach as in the above proof of Lemma \ref{exp}, but we would not need such a result and so we leave the details to the interested reader. The hardest case for Conjecture \eqref{rs} is when $r_0=1$. 
\end{rem}



\section*{Appendix: Linear Strichartz estimates and consequences for local well-posedness}
We review the known linear Strichartz estimates on products of spheres and tori and their consequences for local well-posedness.  Let 
$$\delta(p,d):=\left\{
\begin{array}{ll}
    \frac{d-1}{2}-\frac{d}{p}, &\frac{2(d+1)}{d-1}\leq p\leq\infty, \\
     \frac{d-1}{2}(\frac12-\frac1p),& 2\leq p\leq   \frac{2(d+1)}{d-1}.
\end{array}
\right.$$
Let $M=\mathbb{S}^{d_1}\times\mathbb{S}^{d_2}\times\cdots\times \mathbb{S}^{d_{r_0}}\times \mathbb{T}^{r_1}$ with $r=r_0+r_1\geq 2$. Denote 
$$d'=\min\{d_i: i=1,\ldots,r_0, \ d_i\neq 2\},$$
with the understanding that $\min\emptyset=\infty$.

\begin{thm}
Let $I$ be a finite time interval. Then the Strichartz estimate
\begin{align}\label{strichartz}
\|e^{it\Delta} f\|_{L^{p}(I, L^q(M))}\lesssim \|f\|_{H^s(M)},
\end{align}
holds for the following cases:\\
(1) $q=2$, $s\geq 0$. \\
(2) $p\geq 2$, $q<\infty$, $\frac2p+\frac dq=\frac d2$, $s\geq \frac1p$.\\
(3) $p=q\geq 2$, $s\geq \gamma(p)+\sum_{i=1}^{r_0}\delta(p,d_i)$, 
where $\gamma(p)$ is any exponent for which it holds  
\begin{align}\label{restrictedStrichartz}
\left\|\sum_{\substack{\xi=(\xi_0,\xi_1)\in\mathbb{Z}^{r_0}_{\geq 0}\times\mathbb{Z}^{r_1}
\\ |\xi|\leq N
}}a_\xi e^{-it|\xi|^2+i\langle x_1,\xi_1\rangle}\right\|_{L^p_{t,x_1}([0,2\pi]^{1+r_1})}\lesssim N^{\gamma(p)}\|a_\xi\|_{l^2(\mathbb{Z}^r)}.
\end{align}
In particular, due to Lemma \ref{exp}, (i) of Lemma \ref{e1}, and Remark \ref{remcounting}, we have 
$$\gamma(p)=\left\{
\begin{array}{ll}
    \frac r2-\frac{2}{p}, &\text{ if } p>2,\ r_1=0, \\
      \frac{r}2-\frac{2+r_1}{p}, &\text{ if }p>\frac{2(r+2)}{r}, \\
      \frac{r}2-\frac{2+r_1}{p}+\varepsilon, &\text{ if }p\geq 2,\ r_0\geq 2,\\ 
  0 & \text{ if }p=2,\ r_0=0,1, 
\end{array}
\right.$$
where $\varepsilon$ can be any positive number. \\
(4) $p=q\geq 2+\frac 8r$, $s\geq \frac{d}2-\frac{d+2}p$.\\
(5) $p=q\geq 2+\frac {4(d'+1)}{d'r}$, $s\geq \frac{d}2-\frac{d+2}p$, provided $r_1=0$ and $d_i$ is odd for all $i=1,\ldots,r_0$. 
\end{thm}
\begin{proof}
(1) is trivial. (2) is the Strichartz estimate of \cite{BGT04} which is true on any compact manifold. 
Now we explain (3). Let $M_0:=\mathbb{S}^{d_1}\times\cdots\times \mathbb{S}^{d_{r_0}}$. For a joint eigenfunction $g\in L^2(M_0)$ of the Laplace--Beltrami operators $\Delta_i$ on $\mathbb{S}^{d_i}$ ($i=1,\ldots,r_0$) of eigenvalue $-n_i^2$ ($n_i\geq0$) respectively, we have 
$$\|g\|_{L^p(M_0)}\lesssim 
\prod_{i=1}^{r_0} (n_i+1)^{\delta(p,d_i)}
\|g\|_{L^2(M_0)}.$$
This follows by applying Sogge's spectral projector bounds \cite{Sog88} to each factor $\mathbb{S}^{d_i}$ and using Minkowski's inequality, just as how we derived \eqref{mpmi}. Now for $f\in L^2(M)$, by writing $e^{it\Delta}f$ as in \eqref{linearsum}, we can first freeze the variable $x_0\in M_0$ and apply \eqref{restrictedStrichartz}, and then apply Minkowski's inequality again and the above joint eigenfunction bound. This gives (3). (4) is a result of \cite{Zha21} which holds on general compact symmetric spaces of compact type as well as on their products with tori. (5) is a result of \cite{Zha20}. 
\end{proof}

Linear Strichartz estimates are able to provide local well-posedness via:
\begin{prop}\label{StoL}
Let $M$ be a compact manifold of dimension $d$. Consider the \eqref{NLS} posed on $M$ with $k\geq 1$. Let $(p,q,s)=(p_0,q_0,s_0)$ be a triple so that the Strichartz estimate \eqref{strichartz} holds. Suppose that $p_0>2k$. Then the \eqref{NLS} is locally well-posed in $H^s$ for all $s>s_0+\frac d{q_0}$, with the solution space as $C([-T,T],H^s(M))\cap L^{p_0}([-T,T],L^\infty(M))$, $T>0$. 
\end{prop}
\begin{proof}
    A word-by-word modification of the proof of Proposition 3.1 in \cite{BGT04}. 
\end{proof}

\begin{thm}\label{llwp}
Let $M=\mathbb{S}^{d_1}\times\cdots\times\mathbb{S}^{d_{r_0}}\times\mathbb{T}^{r_1}$ with $r=r_0+r_1\geq 2$. Let $r_2$ denote the number of 2-sphere factors in the product. 
Then the \eqref{NLS} posed on $M$ with $k\geq 1$ is locally well-posed at regularity $H^s(M)$ for the following cases.

\begin{enumerate}
    \item $k=1$, $r_2\geq 1$, $r=2,3,4$: $s>\frac d2-\frac12$. 

    \item $k=1$, $r_2\geq 1$, $r\geq 5$: $s>\frac d2-\frac{r}{r+4}$.

    \item $k=1$, $r_2=0$, $r_0\geq 2$: $s>\frac d2-\frac{2}{p_0}$, where $p_0=\min\left\{\frac{2(d'+1)}{d'-1}, 2+\frac8r\right\}$. If $d_i$ is odd for all $i=1,\ldots,r$ and $r_1=0$, then we can upgrade it to 
$p_0=\min\left\{\frac{2(d'+1)}{d'-1}, 2+\frac{4(d'+1)}{d'r}\right\}$.

    \item $k=1$, $r_2=0$, $r_0=1$: $s>\frac d2-\frac{2}{p_0}$, where $p_0=\min\left\{\max\{\frac{2(r+2)}{r}, \frac{2(d'+1)}{d'-1}\}, 2+\frac8r\right\}$. If $d_i$ is odd for all $i=1,\ldots,r$ and $r_1=0$, then we can upgrade it to 
$p_0=\min\left\{\max\{\frac{2(r+2)}{r}, \frac{2(d'+1)}{d'-1}\}, 2+\frac{4(d'+1)}{d'r}\right\}$.


    \item $k=2$, $r\geq 4$:  $s>s_c=\frac d2-\frac12$. 
    \item $k=2$, $r=3$, $r_2=2,3$:  $s>\frac d2-\frac37$. 
    \item $k=2$, $M=\mathbb{S}^2\times\mathbb{S}^2$:  $s>\frac d2-\frac13$.


    \item $k=2$, $r_2= 1$, $r=2$:  $s>\frac d2-\frac{d'}{2(d'+1)}$.

    \item $k=2$, $r_2= 1$, $r=3$, $d'\leq 6$:  $s>\frac d2-\frac37$.

    \item $k=2$, $r_2= 1$, $r=3$, $d'\geq 7$:  $s>\frac d2-\frac{d'}{2(d'+1)}$.

    \item $k=2$, $r_2=0$:  $s>s_c=\frac d2-\frac12$.

    \item $k\geq 3$:  $s>s_c=\frac d2-\frac1k$.

\end{enumerate}







\end{thm}

\begin{proof}
For each case, we optimize the local well-posedness range $s>s_0+\frac d{q_0}$ given in Proposition \ref{StoL} among all Strichartz triples $(p_0,q_0,s_0)$ given in Theorem~\ref{strichartz}. The theorem follows by an easy but tedious case-by-case calculation. 
\end{proof}

\begin{rem}
    Optimal linear Strichartz estimates on a compact manifold, even in the $p=q$ case of \eqref{strichartz}, are highly challenging. The only examples of compact manifolds for which optimal linear Strichartz estimates in the $p=q$ case are known are rectangular tori, as established in \cite{BD15,KV16}, and Zoll manifolds, as recently established in \cite{HS25} up to $\varepsilon$-losses.
\end{rem}

\end{document}